\documentclass[12pt]{article}

\usepackage[T1]{fontenc}
\usepackage[margin = 1in]{geometry}
\usepackage{setspace}

\usepackage{amsthm, amssymb, amstext}
\usepackage[fleqn]{amsmath}
\usepackage{latexsym}
\usepackage[dvips]{graphicx}
\usepackage{comment}
\usepackage{hyperref}
\usepackage[capitalize]{cleveref}
\usepackage{mathtools}
\usepackage{enumerate} 
\usepackage{paralist}
\usepackage{enumitem}
\usepackage{bm}

\usepackage{todonotes}
\usepackage{comment}
\usepackage{tikz}
\usetikzlibrary{graphs, graphs.standard, positioning}
\usetikzlibrary{decorations.pathmorphing}

\crefname{claim}{Claim}{Claims}
\crefname{conjecture}{Conjecture}{Conjectures}
\crefname{figure}{Figure}{Figures}

\newcommand{\children}{\textsf{children}}
\newcommand{\parent}{\textsf{parent}}
\newcommand{\anc}{\textsf{ancestors}}
\newcommand{\desc}{\textsf{descendants}}

\makeatletter
\newtheorem*{rep@theorem}{\rep@title}
\newcommand{\newreptheorem}[2]{%
\newenvironment{rep#1}[1]{%
 \def\rep@title{#2 \ref{##1}}%
 \begin{rep@theorem}}%
 {\end{rep@theorem}}}
\makeatother

\newreptheorem{theorem}{Theorem}
\newreptheorem{lemma}{Lemma}
\newreptheorem{corollary}{Corollary}

\newtheorem{theorem}{Theorem}[section]
\newtheorem{lemma}[theorem]{Lemma}

\newtheorem{observation}{Observation}

\newtheorem{conjecture}{Conjecture}
\newtheorem{meta-conjecture}{Meta-conjecture}
\newtheorem{meta-problem}{Meta-problem}

\newtheorem{problem}{Problem}
\newtheorem{fact}{Fact}

\theoremstyle{definition}

\newcommand\tw{\text{tw}}
\newcommand\tww{\text{tww}}
\newcommand\otww{\text{otww}}
\newcommand\bn{\text{bn}}
\newcommand\sn{\text{sn}}

\renewcommand\geq{\geqslant}

\definecolor{myRed}{rgb}{0.68, 0.05, 0.0}
\colorlet{myBlue}{blue!70!black}
\colorlet{myViolet}{myBlue!55!myRed}
\definecolor{darkraspberry}{rgb}{0.53, 0.15, 0.34}
\definecolor{olive}{rgb}{0.42, 0.56, 0.14}

\hypersetup{
	colorlinks=true,
        linkcolor=myBlue, 
        citecolor=myBlue,
	bookmarksopen=true,
	bookmarksnumbered,
	bookmarksopenlevel=2,
	bookmarksdepth=3
}

\usepackage{authblk}

\title{Every Graph is Essential to Large Treewidth}

\author[1]{Bogdan Alecu}
\author[2]{\'Edouard Bonnet}
\author[1]{\\Pedro Bureo Villafana}
\author[2]{Nicolas Trotignon}

\affil[1]{School of Computer Science, University of Leeds, Leeds, LS2 9JT, UK}
\affil[2]{CNRS, ENS de Lyon, Université Claude Bernard Lyon 1, LIP,
  UMR 5668, Lyon, France}

\date{}

\begin{document}

\maketitle

\begin{abstract}
  We show that for every graph $H$, there is a~hereditary weakly sparse graph class $\mathcal C_H$ of unbounded treewidth such that the $H$-free (i.e., excluding $H$ as an induced subgraph) graphs of $\mathcal C_H$ have bounded treewidth.
  This refutes several conjectures and critically thwarts the quest for the unavoidable induced subgraphs in classes of unbounded treewidth, a~wished-for induced counterpart of the Grid Minor theorem.
  We actually show a~stronger result: For every positive integer $t$, there is a~hereditary graph class $\mathcal C_t$ of unbounded treewidth such that for any graph $H$ of treewidth at~most~$t$, the $H$-free graphs of $\mathcal C_t$ have bounded treewidth. 
  Our construction is a variant of so-called \emph{layered wheels}.

  We also introduce a framework of abstract layered wheels, based on their most salient properties.
  In particular, we streamline and extend key lemmas previously shown on individual layered wheels.
  We believe that this should greatly help develop this topic, which appears to be a~very strong yet underexploited source of counterexamples. 
\end{abstract}

\section{Introduction}

A~possible reading of the Grid Minor theorem of Robertson and Seymour~\cite{ROBERTSON198692} is that for every class $\mathcal C$ of unbounded treewidth there is a~family $\mathcal F$ of subdivided walls of unbounded treewidth, such that every graph of $\mathcal F$ is a~subgraph of some graph in~$\mathcal C$.    
Said informally, this identifies the subdivided walls as the unavoidable subgraphs witnessing large treewidth.
Recently, some effort has been put into unraveling the unavoidable induced subgraphs of classes of unbounded treewidth; see for instance the series ``\emph{Induced subgraphs and tree decompositions}''~\cite{istd-series}.

In this line of research, the holy grail would be a~hereditary family $\mathcal F^*$ (analogous to the class of all subdivided walls in the subgraph case), ideally comprising a~few canonical families of unbounded treewidth, such that for every class $\mathcal C$ of unbounded treewidth there is a~subfamily $\mathcal F \subseteq \mathcal F^*$ of unbounded treewidth with the property that every member of~$\mathcal F$ is an induced subgraph of some graph in~$\mathcal C$. 
A~first inspection reveals that $\mathcal F^*$ has to contain (an infinite subfamily of) all complete graphs, all complete bipartite graphs, all subdivided walls, and all the line graphs of subdivided walls. 
This however is not enough.
Quickly, several incomparable families were discovered that had to be added to $\mathcal F^*$: the \emph{layered wheel} of Sintiari and Trotignon~\cite{layered-1}, the \emph{Pohoata--Davies grid}~\cite{Pohoata14,Davies22}, the \emph{death star} of~Bonamy et al.~\cite{Bonamy24}.
One can still entertain hopes that all these constructions can be unified in a~common, and relatively descriptive family.

By generalizing the layered wheel construction, we put an end to this endeavor by showing that no family $\mathcal F^*$ can actually work, except the class of all graphs.
A~class is \emph{weakly sparse} if it excludes, for some finite integer $t$, the biclique $K_{t,t}$ as a~subgraph.\footnote{All the other notions used without a~definition in this introduction are defined in~\cref{sec:prelim}.}

\begin{theorem}\label{thm:main-one-graph}
  For every graph $H$, there is a~hereditary weakly sparse class~$\mathcal C_H$ of unbounded treewidth such that the subclass of $H$-free graphs of~$\mathcal C_H$ has bounded treewidth.
\end{theorem}

Hence, if there was a~graph $H$ not in $\mathcal F^*$, the class~$\mathcal C_H$ of~\cref{thm:main-one-graph} would not admit any desirable subfamily $\mathcal F \subseteq \mathcal F^*$.
We actually show a~significantly stronger result than~\cref{thm:main-one-graph}.

\begin{theorem}\label{thm:main-bdd-tw}
   For every positive integer $t$, there is a~hereditary (weakly sparse) class~$\mathcal C_t$ of unbounded treewidth such that for any graph $H$ of treewidth at~most~$t$, the subclass of $H$-free graphs of $\mathcal C_t$ has bounded treewidth.
\end{theorem}

With a~minor adjustment to the construction for~\cref{thm:main-bdd-tw}, we obtain the following triangle-free variation.

\begin{theorem}\label{thm:main-triangle-free}
   For every positive integer $t$, there is a~hereditary (weakly sparse) triangle-free class~$\mathcal C'_t$ of unbounded treewidth such that for any triangle-free graph $H$ of treewidth at~most~$t$, the subclass of $H$-free graphs of $\mathcal C'_t$ has bounded treewidth.
\end{theorem}

One can observe that a~hereditary class satisfying~\cref{thm:main-bdd-tw} (or \cref{thm:main-triangle-free}) has to be weakly sparse (in contrast with \cref{thm:main-one-graph} when $H$ is a~complete graph or a~complete bipartite graph). 

In one sweep, \cref{thm:main-bdd-tw} refutes several conjectures stating that every hereditary (weakly sparse) class of unbounded treewidth contains induced subgraphs of unbounded treewidth with some additional property.
Here are some examples.

\begin{conjecture}[Hajebi, Conjecture 1.15 of~\cite{Hajebi24}, previously refuted in~\cite{chudnovsky2024treewidthmaximumcliques}]\label{conj:Hajebi1}
  Every hereditary weakly sparse class of unbounded treewidth contains a~subclass of 2-degenerate graphs of unbounded treewidth.
\end{conjecture}

\begin{conjecture}[Hajebi, Conjecture 1.14 of~\cite{Hajebi24}, previously refuted in~\cite{chudnovsky2024treewidthmaximumcliques}]\label{conj:Hajebi2}
  Every hereditary weakly sparse class of unbounded treewidth contains a~subclass of clique number at~most~4 (resp.\ at~most~$c$ for some fixed integer $c \geqslant 4$) and unbounded treewidth.
\end{conjecture}

\begin{conjecture}[Trotignon]\label{conj:string-graphs}
  Every hereditary class of unbounded treewidth contains a~subclass of string graphs of unbounded treewidth or a~$K_{\ell,\ell}$ induced minor for every~$\ell$.
\end{conjecture}

Applying \cref{thm:main-one-graph} with $H$ being the 4-vertex clique, and 5-vertex clique (resp.~$(c+1)$-vertex clique) refutes Conjectures~\ref{conj:Hajebi1} and~\ref{conj:Hajebi2}, respectively.
Applying \cref{thm:main-bdd-tw} with $t=4$ refutes~\cref{conj:string-graphs} as graphs of treewidth at~most~4 includes the \mbox{1-subdivision} of the \mbox{5-vertex} clique and $K_4$.
The former graph ensures that there is no subclass of~$\mathcal C_t$ both of string graphs and of unbounded treewidth.
Besides, $\mathcal C_t$ does not contain as induced subgraphs (line graphs of) subdivided walls of arbitrarily large treewidth (as these graphs are $K_4$-free), or equivalently $\mathcal C_t$ does not contain arbitrarily large walls (or grids) as induced minors~(see for instance~\cite{ABOULKER2021103394}).
Thus, $\mathcal C_t$ cannot have $K_{\ell, \ell}$ induced minors for arbitrarily large~$\ell$, since a~result of Chudnovsky et al.~\cite{chudnovsky2024cbim} implies that such classes have  arbitrarily large walls as induced minors or $K_4$-free induced subgraphs of arbitrarily large treewidth.

Chudnovsky and Trotignon~\cite{chudnovsky2024treewidthmaximumcliques} note that the class they use to refute \cref{conj:Hajebi2} with $c \geqslant 4$ has clique number $c+2$, and ask if this can be done with a~class of clique number $c+1$; see~Question 1.8.
Our new construction achieves this as a~byproduct, since the clique number of the class $\mathcal C_t$ of \cref{thm:main-bdd-tw} is $t+1$ (see~\cref{obs:Ct-clique-number}).

Rose McCarty discovered a~construction of a~weakly sparse family of unbounded treewidth of outerstring graphs (also discovered independently in~\cite{Hickingbotham-circle} and~\cite{sparseOuterString}), and together with the fourth author deduced from it that the then-existing layered wheels as well as the Pohoata--Davies grid are string graphs, while the death star admits $K_{\ell,\ell}$ induced minors for arbitrarily large~$\ell$.
This motivated~\cref{conj:string-graphs}.
We also note that Chudnovsky, Fischer, Hajebi, Spirkl, and Walczak disprove a~variant of~\cref{conj:string-graphs} asking for a~subclass of triangle-free outerstring graphs instead~\cite{outerstring-conjecture}; which~\cref{thm:main-triangle-free} more generally refutes with string graphs instead of outerstring graphs.

Let us unify the previous conjectures under (two variants of) a~generalizing meta-conjecture.

\begin{meta-conjecture}[Meta-conjecture$(\Pi)$]\label{meta-conj:plain}
  Every hereditary class $\mathcal C$ of unbounded treewidth contains a~subclass $\mathcal C' \subseteq \mathcal C$ of unbounded treewidth with property $\Pi$.
\end{meta-conjecture}

\begin{meta-conjecture}[Meta-conjecture-ws$(\Pi)$]\label{meta-conj:ws}
  Every hereditary weakly sparse class $\mathcal C$ of unbounded treewidth contains a~subclass $\mathcal C' \subseteq \mathcal C$ of unbounded treewidth with property $\Pi$.
\end{meta-conjecture}

One can see that~\cref{conj:Hajebi1,conj:Hajebi2} are particular cases of~\cref{meta-conj:ws}, and \cref{conj:string-graphs}, of~\cref{meta-conj:plain}.
We observe that Meta-conjecture($\Pi$) straightforwardly implies Meta-conjecture-ws($\Pi$).
Note also that we do not require $\mathcal C'$ to be itself hereditary, although for \emph{some} properties $\Pi$ this could be ensured for free.

Let us say that a~property $\Pi$ is \emph{finitely-hereditary} if there is a~(finite) graph $H$ such that if $H$ is an induced subgraph of some $G \in \mathcal C'$, then $\mathcal C'$ does not satisfy $\Pi$.
For instance, the properties of being 2-degenerate, having clique number at~most~4, or being a~string graph are all finitely-hereditary (due to the abovementioned respective graphs~$H$).
A~weaker assumption on $\Pi$ is that it is \emph{treewidth-hereditary}, i.e., that there is a~(finite) integer~$t$ such that if $\mathcal C'$ contains as (induced) subgraphs every graph of treewidth at~most~$t$ then $\mathcal C'$ does not satisfy~$\Pi$.

One can then reformulate \cref{thm:main-one-graph}, and its strengthening \cref{thm:main-bdd-tw}, this way.

\begin{reptheorem}{thm:main-one-graph}
  For every finitely-hereditary property $\Pi$, Meta-conjecture-ws$(\Pi)$ and thus Meta-conjecture$(\Pi)$ are refuted.
\end{reptheorem}

\begin{reptheorem}{thm:main-bdd-tw}
  For every treewidth-hereditary property $\Pi$, Meta-conjecture-ws$(\Pi)$ and thus Meta-conjecture$(\Pi)$ are refuted.
\end{reptheorem}

Thus in any further attempt to find properties forced by large treewidth, one has to consider properties $\Pi$ that are not finitely-hereditary nor, more generally, treewidth-hereditary.
A~result of Bonnet~\cite{Bonnet24} overcomes the very barrier of~\cref{thm:main-one-graph,thm:main-bdd-tw} by relying on the non-hereditary parameter of average degree (or edge density): for every $\varepsilon > 0$, every hereditary weakly sparse class of unbounded treewidth admits a~(non-hereditary) subclass of unbounded treewidth and average degree at~most~$2+\varepsilon$.
Recall indeed that we do not require the subfamily $\mathcal C'$ to be hereditary. 

Another property that is not treewidth-hereditary is that of having bounded twin-width.\footnote{Actually \emph{bounded treewidth} is also not treewidth-hereditary, but the associated meta-conjectures are trivially false.}
We thus give the following special case of~\cref{meta-conj:plain}, not refuted by our present work, and motivated by the fact that the weakly sparse layered wheels, the Pohoata--Davies grids (and their extensions), and the death star all have bounded twin-width.
\begin{conjecture}\label{conj:tww-tw}
  Every hereditary class of unbounded treewidth admits a~subclass of unbounded treewidth and bounded twin-width.
\end{conjecture}
We will come back to~\cref{conj:tww-tw}.

Say that a~graph $H$ is \emph{essential} if there is a~hereditary class $\mathcal C_H$ of unbounded treewidth such that the class $\{G \in \mathcal C_H~:~G \text{ is } H\text{-free}\}$ has bounded treewidth.
Call a family $\mathcal H$ of graphs \emph{essential} if there is a~hereditary class $\mathcal C_{\mathcal H}$ of unbounded treewidth such that, for every $H \in \mathcal H$, the class $\{G \in \mathcal C_{\mathcal H}~:~G \text{ is } H\text{-free}\}$ has bounded treewidth.
In this terminology, we can restate a~slightly weaker form of~\cref{thm:main-one-graph} and equivalent form of~\cref{thm:main-bdd-tw}.

\begin{theorem}\label{thm:main-one-graph-ess}
    Every graph is essential.  
\end{theorem}

\begin{theorem}
    \label{thm:main-bdd-tw-ess}
    For every $t \in \mathbb N$, the family of graphs of treewidth at~most~$t$ is essential.  
\end{theorem}

Our proof of~\cref{thm:main-bdd-tw} (or \cref{thm:main-bdd-tw-ess}) revisits a~variant of the layered wheels by Chudnovsky and Trotignon~\cite{chudnovsky2024treewidthmaximumcliques}, makes it more general, and abstracts out its properties. 

\subsection{Layered-wheelology}

Several constructions have been called \emph{layered wheels}: two constructions in~\cite{layered-1} and several variants in~\cite{chudnovsky2024treewidthmaximumcliques}.
In order to facilitate further research, such as refuting Meta-conjecture$(\Pi)$ and Meta-conjecture-ws$(\Pi)$ for some properties $\Pi$ that are not treewidth-hereditary, we abstract out mandatory and optional properties of the layered wheels.

A~\emph{layered wheel} is a~(countably infinite) graph $G$ on the same vertex set as a~countably infinite \emph{locally-finite} (i.e., every node has finite degree) planarly-embedded rooted tree $T$ such that the following conditions hold.
\medskip
\begin{compactenum}
\item \label{lw:1} For every natural number $n$, the set of nodes at distance $n$ from the root in $T$, $L_n$, induces in $G$ a~path\footnote{In previous layered wheels, $L_n$ sometimes induced a~cycle rather than a~path, but this came without any functional differences.} that goes ``left-to-right'' in the planar embedding of~$T$.
  Each set $L_n$ is called a \emph{layer}, the edges induced by a~layer are called \emph{layer edges}, and the set of all layer edges is denoted by $E_L$. 
\item \label{lw:2} Every edge in $E(G) \setminus E_L$ is between a~pair of ancestor--descendant of~$T$.
\item \label{lw:3} $T$ has no arbitrarily long paths of vertices of degree 2 (in $T$).
\end{compactenum}

\medskip
We may refer to~Condition~\ref{lw:1} as the \emph{layer condition} and to~Condition~\ref{lw:2} as the \emph{treedepth condition} (since this is exactly the requirement of treedepth decompositions).
Condition~\ref{lw:3} is very mild, but forces $T$ to branch.
We note that the so-called \emph{($K_4$, even hole)-free layered wheel} from~\cite{layered-1} does not satisfy the treedepth condition.
Thus we propose a~weakening of Condition~\ref{lw:2}.

For any non-negative integer $t$, a~\emph{$t$-stroll} between $u \in V(T)$ and $v \in V(T)$ is a~$u$--$v$ path~$P$ in the graph $(V(T),E(T) \cup E_L)$ such that $P$ has no $t+1$ consecutive layer edges, and the intersection of every layer with $P$ is consecutive in~$P$ (i.e., a~stroll cannot reenter a~previously visited layer).
Two nodes $u,v \in V(T)$ are in a~\emph{$t$-wide ancestor--descendant} relation if there is a~$t$-stroll between $u$ and~$v$.
Then we define:
\medskip
\begin{compactitem}
 \item[(2')] \phantomsection\label{lw:2p} There is a~finite integer~$t \geqslant 0$, such that every edge in $E(G) \setminus E_L$ is between a~pair of $t$-wide ancestor--descendant of~$T$.
\end{compactitem}
\medskip
A \emph{generalized layered wheel} is one that satisfies Conditions \ref{lw:1}, \hyperref[lw:2p]{(2')}, and \ref{lw:3}.
The ($K_4$, even hole)-free layered wheel of Sintiari and Trotignon satisfies Condition~\hyperref[lw:2p]{(2')} with $t=1$.
Note that Condition~\hyperref[lw:2p]{(2')} with $t=0$ coincides with Condition~\ref{lw:2}. 

We now define some further optional properties:
\medskip
\begin{compactenum}
\setcounter{enumi}{3}
\item \label{lw:4} For every $i \neq j \in \mathbb N$, there is at least one edge from $L_i$ to $L_j$. 
$G$ is then called \emph{proper}.
\item \label{lw:5} $T$ has no leaf. (Every maximal branch from the root is infinite.)
$G$ is then called \emph{neat}.
\item \label{lw:6} There is an~integer $d$ such that every node of~$T$ has at~most~$d$ children in~$T$.
$G$ is then called \emph{bounded}, more precisely \emph{$d$-bounded}.
\item \label{lw:7}  Every node in every~$L_n$ has at~most~$f(n)$ children in~$T$, for some particular function~$f \colon \mathbb N \to \mathbb N$.
$G$ is then called \emph{$f$-bounded}.
\item \label{lw:8} There is an~integer $t$ such that for every node $v$ of $T$, there is a~subset $X_v$ of ancestors of $v$ of size at~most~$t$ with the property that in $G-E_L$ every edge with exactly one endpoint within $v$ and its descendants has its other endpoint in~$X_v$.
$G$ is then called \emph{upward-restricted}, more precisely \emph{$t$-upward-restricted}.
\item \label{lw:9} For every node $v$ of~$T$, the ancestors of~$v$ that are adjacent to~$v$ in $G$ form a~clique in~$G$.
$G$ is then called \emph{upward-simplicial}.
\item \label{lw:10} For every triple of pairwise distinct nodes $u, v, w$, if $u$ is an ancestor of $v$, $v$ is an ancestor of~$w$, and $uw \in E(G)$, then $uv \in E(G)$. 
$G$ is then called \emph{upward-nested}.
\end{compactenum}

\medskip

Note that, as a~slight abuse of language, the above properties attribute to $G$ what is actually a~property of the pair $(G,T)$.
\Cref{fig:abstract-lw} shows the first few layers of a~layered wheel $G$ with rooted tree $T$, which can be completed to a~proper, neat, 3-bounded layered wheel, but which is not \emph{upward-simplicial} nor \emph{upward-nested}.
\begin{figure}[h!]
\centering
  \begin{tikzpicture}[
  vertex/.style={fill,circle,inner sep=0.05cm, minimum width=0.05cm},
  tree-edge/.style={dotted},
  layer-edge/.style={blue, very thick},
  edge/.style={very thick},
  tedge/.style={white, line width=0.6pt, dotted}
  ]
  \def\s{1}
  \foreach \l/\x/\y in {
  a1/7/5,
  b1/3/4, b2/7/4, b3/11/4,
  c1/1/3, c2/3/3, c3/4.5/3,  c4/6/3, c5/7/3, c6/8/3,  c7/10/3, c8/12/3,
  d1/0/2, d2/1/2, d3/2/2,  d4/3/2,  d5/4/2, d6/5/2,  d7/5.5/2, d8/6.5/2,  d9/7/2,  d10/7.5/2, d11/8.5/2,  d12/9.4/2, d13/10/2, d14/10.6/2,  d15/11.4/2, d16/12.6/2,
  e1/-0.5/1, e2/0/1, e3/0.5/1,  e4/1/1,  e5/1.5/1, e6/2/1, e7/2.5/1,  e8/3/1,  e9/3.6/1, e10/4.4/1,  e11/5/1,  e12/5.35/1, e13/5.75/1, e14/6.25/1, e15/6.65/1, e16/7/1,
  e17/7.35/1, e18/7.65/1,  e19/8/1, e20/8.5/1, e21/9/1,  e22/9.4/1,  e23/9.75/1, e24/10.25/1,  e25/10.6/1,  e26/11/1, e27/11.4/1, e28/11.8/1,  e29/12.2/1, e30/12.6/1, e31/13/1}
  {
   \node[vertex] (\l) at (\x * \s, \y * \s) {} ;
  }

  \foreach \i/\j in {
  a1/b1, a1/b2, a1/b3,
  b1/c1, b1/c2, b1/c3,
  b2/c4, b2/c5, b2/c6,
  b3/c7, b3/c8,
  c1/d1, c1/d2, c1/d3,
  c2/d4,
  c3/d5, c3/d6,
  c4/d7, c4/d8,
  c5/d9,
  c6/d10, c6/d11,
  c7/d12, c7/d13, c7/d14,
  c8/d15, c8/d16,
  d1/e1, d1/e2, d1/e3,
  d2/e4,
  d3/e5, d3/e6, d3/e7,
  d4/e8,
  d5/e9, d5/e10,
  d6/e11,
  d7/e12, d7/e13,
  d8/e14, d8/e15,
  d9/e16,
  d10/e17, d10/e18,
  d11/e19, d11/e20, d11/e21,
  d12/e22,
  d13/e23, d13/e24,
  d14/e25,
  d15/e26, d15/e27, d15/e28,
  d16/e29, d16/e30, d16/e31}
  {
  \draw[tree-edge] (\i) -- (\j) ;
  }

  \foreach \l/\i in {b/2, c/7, d/15, e/30}{
    \foreach \k [count = \kp from 2] in {1,...,\i}{
      \draw[layer-edge] (\l\k) -- (\l\kp) ;
    }        
  }
  \foreach \i/\j in {0/6, 1/2.4, 2/0.5, 3/-0.45, 4/-0.9}{
    \node at (\j * \s, 5 * \s - \i * \s) {$L_\i$} ;
  }

  \foreach \i/\j in {a1/b1, b1/c1, b1/c2, c1/d1, c1/d3, c3/d5, c3/d6, d1/e2, d3/e5, d3/e6, d3/e7, d4/e8, d5/e10, d6/e11, 
  a1/b2, b2/c4, b2/c6, c4/d8, c5/d9, c6/d11, d7/e12, d7/e13, d8/e15, d9/e16, d10/e17, d10/e18, d11/e20, d11/e21}{
    \draw[edge] (\i) -- (\j) ;
    \draw[tedge] (\i) -- (\j) ;
  }

  \foreach \i/\j in {a1/c2, a1/d1, a1/d3, a1/e4, a1/e9, b1/d3, b1/e4, b1/e11, c1/e2, c1/e3, c1/e6, c3/e11, 
  b2/d7, b2/d8, b2/d11, c6/e17, c6/e18, c6/e20, b2/e13, a1/c4, a1/d10, a1/e19}{
    \draw[edge] (\i) -- (\j) ;
  }

  \foreach \i/\j/\b in {a1/d9/30, c2/e8/30, a1/e21/15}{
    \draw[edge] (\i) to [bend left=\b] (\j) ;
  }
  \end{tikzpicture}
  \caption{Possible first five layers of a~proper, neat, 3-bounded layered wheel $G$ with rooted tree~$T$.
  The edges of $T$ are dotted, the layer edges of $G$ are in blue, and its other edges are in black (with white dots if they are also edges of~$T$).}
  \label{fig:abstract-lw}
\end{figure}
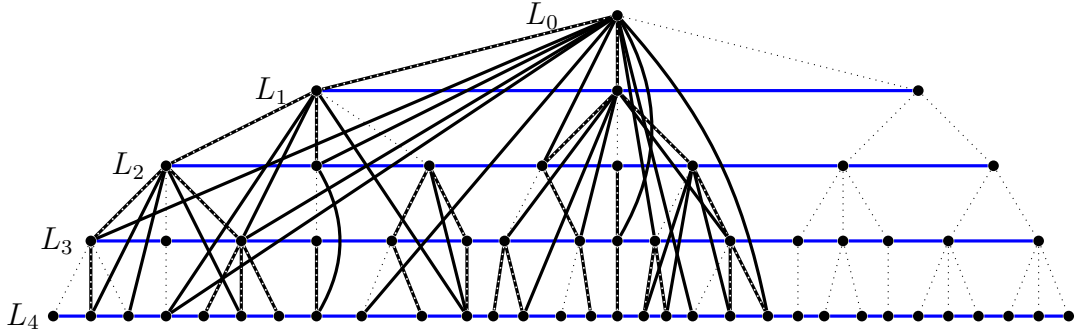
Note that the class yielded by a~proper layered wheel has unbounded treewidth as it contains every clique as a~minor (just contract each layer to a~single vertex).
The proof of~\cref{thm:main-bdd-tw} consists of building a~particular proper, neat, bounded, upward-restricted layered wheel.

All layered wheels built so far are proper and neat.
The so-called \emph{(theta, triangle)-free layered wheels} of Sintiari and Trotignon~\cite{layered-1} are further $d$-bounded (where $d \geq 17$ can be adjusted to augment the girth), upward-restricted, upward-simplicial, and upward-nested.
The so-called \emph{(even hole, $K_4$)-free layered wheels}~\cite{layered-1} are similar, except that, as already mentioned, they satisfy Condition~\hyperref[lw:2p]{(2')} instead of Condition~\ref{lw:2}.  
The layered wheels of Chudnovsky and Trotignon~\cite{chudnovsky2024treewidthmaximumcliques} are $f$-bounded for well-chosen functions~$f$, upward-simplicial, and upward-nested (but not always upward-restricted).

\medskip
A family $\mathcal F$ of induced subgraphs of~$G$ satisfies the \emph{bounded-branch} property if there is an~integer~$h$ such that for every graph $G' \in \mathcal F$ and for every $v \in V(T)$ (not necessarily in $V(G')$), there is a~\emph{downward path} in $T$ starting at~$v$ (i.e., a~maximal, possibly infinite, path starting at $v$, and iteratively visiting a~child of the current node while the current node is not a~leaf) that contains at~most~$h$ vertices of~$G'$.
Note that in a~neat layered wheel, every downward path is infinite.
Here is a~crucial feature of upward-restricted layered wheels. 

\begin{theorem}\label{thm:small-tw-bbp}
  For every neat, upward-restricted layered wheel $W$, every class of finite graphs satisfying the bounded-branch property in $W$ has bounded treewidth.
\end{theorem}

\Cref{thm:small-tw-bbp} is essentially proved in~\cite{chudnovsky2024treewidthmaximumcliques}, albeit not under this formalism.
We shall then simply see how to engineer a~neat, proper, upward-restricted layered wheel $W_t$ such that the absence of any (a~priori unknown) graph of treewidth at~most~$t$ in finite induced subgraphs of~$W_t$ ensures the bounded-branch property.

\Cref{thm:small-tw-bbp} further leads to the following theorem.

 \begin{theorem}\label{thm:lw-log-tw-ub}
   The $n$-vertex induced subgraphs of any neat, upward-restricted layered wheel have treewidth $O(\log n)$.
 \end{theorem}

\Cref{thm:lw-log-tw-ub} unifies and extends what was individually observed on the finite induced subgraphs of the (theta, triangle)-free layered wheel~\cite{layered-1}, and on the $K_t$-free finite induced subgraphs of the layered wheels from~\cite{chudnovsky2024treewidthmaximumcliques}.
Indeed, one can observe that the $K_t$-free finite induced subgraphs of an upward-simplicial, upward-nested layered wheel have the property of upward-restriction. 
On the other hand, it is not difficult to see that any proper, bounded layered wheel admits arbitrarily large $n$-vertex induced subgraphs of treewidth $\Omega(\log n)$.
Thus a~source of classes with logarithmic treewidth is given by considering the finite induced subgraphs of any proper, neat, bounded, upward-restricted layered wheel.

\medskip

Someone interested in refuting~\cref{conj:tww-tw} by developing further layered wheels should be aware of the following obstacle. 
\begin{theorem}\label{thm:tww-of-ur-lw}
The class of finite induced subgraphs of any neat, upward-restricted layered wheel has bounded twin-width.
\end{theorem}
The \emph{neat} condition mainly simplifies the proof of~\cref{thm:tww-of-ur-lw}, and should not be needed.
Thus, one would have to drop the upward-restricted condition, and rely on other specificities of one's layered wheel to bound the treewidth of subclasses $\mathcal C'$ of bounded twin-width.

\subsection{Related work and further directions}

We identify a~few directions related to the current work.

\medskip
\textbf{Essentiality for other graph parameters.}
A~reader familiar with the paper ``\emph{Induced subgraphs of induced subgraphs of large chromatic number}'' of Gir\~ao et al.~\cite{giraoetal} likely noticed a~parallel with~\cref{thm:main-one-graph}.
The authors show that for every graph $H$, there is a~class $\mathcal C$ of unbounded chromatic number such that the \mbox{$H$-free} graphs of $\mathcal C$ have bounded chromatic number. 
(Furthermore, if $H$ has at least one edge, the class $\mathcal C$ can be picked to have the same clique number as~$H$.)
In this sense, the current paper can be thought of as ``\emph{Induced subgraphs of induced subgraphs of large treewidth}.''
We note however that the construction in~\cite{giraoetal} contains, by design, arbitrarily large induced bicliques, so could not guide us in achieving~\cref{thm:main-one-graph}.

We now know that no fixed induced subgraph can be removed from every class $\mathcal C$ of unbounded $p$ while preserving that $p$ is unbounded, for parameter $p$ equal to chromatic number or treewidth.  
Other graph parameters can be considered such as clique-width, twin-width, etc.
As our construction leads to a~weakly sparse class, within which treewidth and clique-width are known to be functionally equivalent~\cite{Gurski00}, our paper also offers a~complete answer for clique-width.
The case of twin-width remains interesting.
It is known that the class of permutation graphs is a~minimal hereditary class of unbounded twin-width~\cite{twin-width1}.
This translates into the essentiality (for twin-width) of every permutation graph.
To our knowledge, the essentiality (for twin-width) of any other graph is open.  
In general, we propose the following questions.

\begin{meta-problem}[Characterize $p$-essential graphs]
  For a~parameter of choice~$p$, which graphs $H$ are such that there is a~hereditary class of unbounded $p$ whose $H$-free graphs have bounded~$p$?
\end{meta-problem}

Observe that Ramsey's theorem can be rephrased as the fact that complete or edgeless graphs are the only $p$-essential graphs when $p$ is the number of vertices.

\medskip

\textbf{Characterizing the essential families.}
For families rather than single graphs, our understanding of essentiality (for treewidth) is not quite complete.
What about families $\mathcal H$ of unbounded treewidth?
Such a~family $\mathcal H$ is essential if and only if its hereditary closure is a~minimal hereditary class of unbounded treewidth.
Thus, any family $\mathcal H$ consisting of infinitely many complete graphs or of infinitely many complete bipartite graphs is essential.
On the other hand, a~family $\mathcal H$ consisting of infinitely many complete graphs (resp.~complete bipartite graphs) plus at~least one graph that is not a~clique (resp.~not a~biclique) is \emph{not} essential.
Hence, we narrowed down the open cases to weakly sparse families.
Is there a~weakly sparse family whose hereditary closure is a~minimal hereditary class of unbounded treewidth?
This is precisely a~question of~Cocks~\cite{COCKS2024104005}, which we reformulate here.

\begin{conjecture}[Cocks's Conjecture 1.5 in~\cite{COCKS2024104005}]
    A family of unbounded treewidth is essential if and only if it contains only complete graphs, or only complete bipartite graphs. 
\end{conjecture}

\medskip

\textbf{Treewidth logarithmically bounded in the number of vertices.}
Sintiari and Trotignon~\cite{layered-1} remarked that their layered wheel constructions have treewidth logarithmic in their number of vertices, and suggested relaxing the bounded treewidth condition and investigating logarithmic treewidth instead. 
Many (NP-hard) problems can be solved in polynomial time in $n$-vertex graphs of treewidth $O(\log n)$, such as every problem expressible in the so-called \emph{Existential Counting Modal Logic} of Pilipczuk, a~large fragment of Monadic Second-Order logic~\cite{Pilipczuk11}.
In the past five years, several classes have been shown to have logarithmic treewidth~\cite{Chudnovsky_2022, Bonamy24, BonnetD23, chudnovsky2024inducedsubgraphstreedecompositions, sparseOuterString}.
Our result shows that any graph is responsible for the transition from unbounded to bounded treewidth in some class.
Is this true for the transition between superlogarithmic and logarithmic treewidth? 

\begin{problem}\label{p:logarithmic-tw}
    For which families $\mathcal H$ is there a hereditary class $\mathcal C$ of superlogarithmic treewidth such that, for every $H \in \mathcal H$, the $H$-free graphs of $\mathcal C$ have at most logarithmic treewidth?
\end{problem}

\Cref{p:logarithmic-tw} is already open for singleton families $\mathcal H = \{H\}$.
It should be noted that the Pohoata--Davies grid~\cite{Pohoata14,Davies22} has no large clique, biclique, subdivided wall, or its line graph as an induced subgraph and has superlogarithmic treewidth: the $n \times n$ Pohoata--Davies grid has treewidth $\Theta(n)$.  

\medskip




\textbf{Essentiality in the high-girth setting.} Given the strong constraints imposed by \cref{thm:main-bdd-tw} on properties forced by large treewidth, another line of investigation is to ask whether the situation is any different in more restricted settings. A natural question in this direction is whether \cref{thm:main-triangle-free} can be generalized to high girth. In \cref{lem:no-high-girth-wheels}, we show that if such a generalization is possible, it cannot use a layered wheel construction, and ask whether the result does in fact generalize (and in particular, whether large treewidth does force certain finitely-hereditary properties in the high-girth setting):

\begin{problem} \label{conj:no-high-girth}
    Does there exist a graph $H$ such that for every $(C_3, C_4)$-free family $\mathcal C$ of unbounded treewidth, the subclass of $H$-free graphs of $\mathcal C$ has unbounded treewidth? 
\end{problem}

\section{Preliminaries}\label{sec:prelim}

For any positive integer $i$, we denote by $[i]$ the set of integers $\{1,2,\ldots,i\}$.

\subsection{Induced subgraphs, neighborhoods, and some special graphs}\label{sec:graph-def}

We denote by $V(G)$ and $E(G)$ the set of vertices and edges of a graph $G$, respectively.
A~graph $H$ is a~\emph{subgraph} of a~graph $G$ if $H$ can be obtained from $G$ by vertex and edge deletions.
Graph~$H$ is an~\emph{induced subgraph} of $G$ if $H$ can be obtained from $G$ by vertex deletions only.
A~graph $G$ is \emph{$H$-free} if $G$ does not contain $H$ as an induced subgraph.
For $S \subseteq V(G)$, the \emph{subgraph of $G$ induced by $S$}, denoted $G[S]$, is obtained by removing from $G$ all the vertices that are not in $S$ (together with their incident edges).
Then $G-S$ is a short-hand for $G[V(G)\setminus S]$.
If $F \subseteq E(G)$, then $G-F$ is the graph $(V(G), E(G) \setminus F)$.
We denote by $N_G(v)$ and $N_G[v]$, the open, respectively closed, neighborhood of $v$ in $G$.
For $S \subseteq V(G)$, we set $N_G(S) := (\bigcup_{v \in S}N_G(v)) \setminus S$ and $N_G[S] := N_G(S) \cup S$.
A~graph $H$ is a~\emph{minor} (resp.~\emph{induced minor}) of a~graph~$G$ if $H$ can be obtained from a~subgraph (resp.~\emph{induced subgraph}) of~$G$ by performing edge contractions.

A~\emph{subdivision} of a~graph $G$ is any graph obtained from $G$ by replacing each edge of~$G$ by a~path of at~least~one edge.
The \emph{$s$-subdivision} of $G$ is the graph obtained from $G$ by replacing each edge of~$G$ by a~path of $s+1$ edges.

The $t$-vertex \emph{complete graph}, denoted by \emph{$K_t$}, is obtained by making adjacent every pair of two distinct vertices among $t$ vertices, and the \emph{complete bipartite graph $K_{t,t}$} with bipartition $(A,B)$ such that $|A|=|B|=t$ is obtained by making every vertex of~$A$ adjacent to every vertex of~$B$. A \emph{clique} of a graph $G$ is a set of vertices inducing a complete graph.
The \emph{clique number} of a~graph $G$ is the largest $t$ such that $K_t$ is an (induced) subgraph of~$G$.
A~graph class $\mathcal C$ is said \emph{weakly sparse} if there is a~finite integer $t$ such that no graph of~$\mathcal C$ contains $K_{t,t}$ as a~subgraph.
A~class~$\mathcal C$ is \emph{hereditary} if for every $G \in \mathcal C$ and every induced subgraph $H$ of $G$, $H \in \mathcal C$.

A~\emph{string graph} is the intersection graph of some collection of (non-self-intersecting) curves in the plane (usually called strings), or equivalently the intersection graph of a~collection of connected sets of some planar graph.
It is known that the 1-subdivision of any non-planar graph is not a~string graph~\cite{Sinden66}.

A~(possibly infinite) graph is \emph{chordal} if it does contain any cycle of length at~least~4 as an induced subgraph.
A~vertex $v$ in a~graph $G$ is \emph{simplicial} if $N_G(v)$ (or $N_G[v]$) is a~clique of~$G$.
The following is a classical fact.
\begin{fact}\label{fact:simpG}
  Every finite chordal graph admits a simplicial vertex. 
\end{fact}

For two positive integers $k, \ell$, the \emph{$k \times \ell$ grid} is the graph on $k\ell$ vertices, say, $v_{i,j}$ with $i \in [k], j \in [\ell]$, such that $v_{i,j}$ and $v_{i',j'}$ are adjacent whenever either $i=i'$ and $|j-j'|=1$ or $j=j'$ and $|i-i'|=1$.
For what comes next, it is helpful to identify vertex $v_{i,j}$ with the point $(i,j)$ of $\mathbb N^2$.
For $k \geqslant 2$ the \emph{$k \times k$ wall} is the subgraph of the $2k \times k$ grid obtained by removing every ``vertical edge'' on an ``even column'' when the edge bottom endpoint is on an ``odd row'', and every ``vertical edge'' on an ``odd column'' when the edge bottom endpoint is on an ``even row,'' and finally by deleting the two vertices of~degree~1 that this process creates.   
See~\cref{fig:grids-walls} for an illustration of the $5 \times 5$ grid and $5 \times 5$ wall.
\begin{figure}[h!]
  \centering
  \begin{tikzpicture}[vertex/.style={draw,circle,inner sep=0.035cm}]
    \def\t{5}
    \pgfmathtruncatemacro\tm{\t-1}
    \pgfmathtruncatemacro\tt{2 * \t}
    \pgfmathtruncatemacro\ttm{\tt - 1}
    \pgfmathtruncatemacro\ht{\t / 2}
    \def\s{0.6}

    \foreach \i in {1,...,\t}{
      \foreach \j in {1,...,\t}{
        \node[vertex] (v\i\j) at (\i * \s, \j * \s) {} ;
      }
    }
    \foreach \i in {1,...,\t}{
      \foreach \j [count = \jm from 1] in {2,...,\t}{
        \draw (v\i\j) -- (v\i\jm) ;
        \draw (v\j\i) -- (v\jm\i) ;
      }
    }

    \begin{scope}[xshift=1.4 * \t * \s cm]
    \foreach \i in {1,...,\tt}{
      \foreach \j in {2,...,\tm}{
        \node[vertex] (w\i-\j) at (\i * \s, \j * \s) {} ;
      }
    }
    \foreach \i in {1,...,\ttm}{
      \node[vertex] (w\i-1) at (\i * \s, \s) {} ;
    }
    \foreach \i in {2,...,\tt}{
      \node[vertex] (w\i-\t) at (\i * \s, \t * \s) {} ;
    }

    \foreach \i in {2,...,\tm}{
      \foreach \j [count = \jm from 1] in {2,...,\tt}{
        \draw (w\j-\i) -- (w\jm-\i) ;
      }
    }
    \foreach \j [count = \jm from 1] in {2,...,\ttm}{
        \draw (w\j-1) -- (w\jm-1) ;
    }
    \foreach \j [count = \jm from 2] in {3,...,\tt}{
        \draw (w\j-\t) -- (w\jm-\t) ;
    }

    \foreach \i in {1,...,\t}{
      \pgfmathtruncatemacro\ii{2 * \i}
      \foreach \j in {1,...,\ht}{
        \pgfmathtruncatemacro\jj{2 * \j}
        \pgfmathtruncatemacro\jjp{\jj + 1}
        \draw (w\ii-\jj) -- (w\ii-\jjp) ;
      }
    }
     \foreach \i in {1,...,\t}{
      \pgfmathtruncatemacro\ii{2 * \i - 1}
      \foreach \j in {1,...,\ht}{
        \pgfmathtruncatemacro\jj{2 * \j - 1}
        \pgfmathtruncatemacro\jjp{\jj + 1}
        \draw (w\ii-\jj) -- (w\ii-\jjp) ;
      }
    }
    \end{scope}
    
  \end{tikzpicture}
  \caption{The $5 \times 5$ grid (left) and the $5 \times 5$ wall (right).}
  \label{fig:grids-walls}
\end{figure}
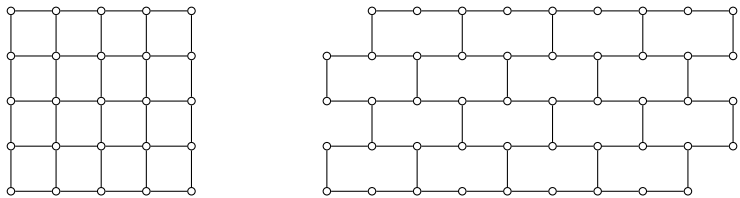

\subsection{Treewidth, brambles, and separation number}\label{sec:tree-dec-bramble}

A~\emph{tree-decomposition} of a~graph $G$ is a~pair $(T,\beta)$ where $T$ is a~tree and $\beta$ is a~map from $V(T)$ to $2^{V(G)}$ satisfying the following conditions:
\begin{compactitem}
\item for every $uv \in E(G)$, there is an~$x \in V(T)$ such that $\{u,v\} \subseteq \beta(x)$, and
\item for every $v \in V(G)$, the set of nodes $x \in V(T)$ such that $v \in \beta(x)$ induces a~non-empty subtree of $T$.
\end{compactitem}
The \emph{width} of $(T,\beta)$ is defined as $\max_{x \in V(T)} |\beta(x)| - 1$, and the \emph{treewidth} of $G$, denoted by $\tw(G)$, is the minimum width of $(T,\beta)$ taken among every tree-decomposition $(T,\beta)$ of~$G$.

The notion of~\emph{bramble} was introduced by Seymour and Thomas~\cite{SeymourT93} as a~min-max dual to treewidth.
A~\emph{bramble} of a graph~$G$ is a~set $\mathcal B := \{B_1, \ldots, B_q\}$ of connected subsets of $V(G)$ such that for every $i, j \in [q]$ the pair $B_i, B_j$ \emph{touch}, i.e., $B_i \cap B_j \neq \emptyset$ or there is some $u \in B_i$ and $v \in B_j$ with $uv \in E(G)$.
A~\emph{hitting set} of~$\{B_1, \ldots, B_q\}$ is a~set $X$ such that for every $i \in [q]$, $X \cap B_i \neq \emptyset$, that is, $X$ intersects every $B_i$.
The \emph{order} of bramble $\mathcal B$ is the minimum size of a~hitting set of $\mathcal B$.

The \emph{bramble number} of $G$, denoted by $\bn(G)$, is the maximum order of a~bramble of~$G$.
The treewidth and bramble number are tied: for every graph $G$, $\tw(G)=\bn(G)-1$~\cite{SeymourT93}.

\medskip

A~subset $S$ of vertices in a~graph~$G$ is an \emph{$\alpha$-balanced separator} if every connected component of $G-S$ has at~most~$\alpha |V(G)|$ vertices.
The \emph{separation number} of a~graph~$G$, denoted by $\sn(G)$, is the smallest integer $s$ such that every (induced) subgraph of~$G$ has a~$\frac{2}{3}$-balanced separator of size at~most~$s$.
Treewidth and separation number are linearly tied.
\begin{theorem}[\cite{DvorakN19}]\label{thm:tw-sn}
  For every graph $G$, $\sn(G) - 1 \leqslant \tw(G) \leqslant 15 \sn(G)$.
\end{theorem}

Small $\frac{2}{3}$-balanced separators are yielded by small(er) $\alpha$-balanced separators with $\alpha < 1$.

\begin{lemma}\label{lem:convert-bs}
  Fix any graph $G$ and $\alpha \in [\frac{2}{3},1)$.
  If every (induced) subgraph of~$G$ admits an $\alpha$-balanced separator of size at~most~$s$, then $G$ has separation number at~most $\frac{\log(2/3)}{\log \alpha} \cdot s$.
\end{lemma}
\begin{proof}
  Take any subgraph $H$ of~$G$, and define $S_i$ as an $\alpha$-balanced separator of size at~most~$s$ of the largest component of $H - \bigcup_{1 \leqslant j < i} S_j$, while this largest component has size more than $\frac{2}{3} |V(H)|$.
  Say that this process eventually defines the sets $S_1, \ldots, S_t \subseteq V(H)$.
  As the size of the largest component is multiplied by a~factor of at~most~$\alpha$ at every iteration, and $\alpha^{\log(2/3)/\log \alpha} = \frac{2}{3}$, we have  $t \leqslant \frac{\log(2/3)}{\log \alpha}$.
  Hence $H$ has a~$\frac{2}{3}$-balanced separator $\bigcup_{i \in [t]} S_i$ of size at~most $st$, so $G$ has separation number at~most~$\frac{\log(2/3)}{\log \alpha} \cdot s$.
\end{proof}

We will rely on the combination of~\cref{thm:tw-sn,lem:convert-bs}.
\begin{lemma}\label{lem:tw-sn-convert}
  Fix any graph $G$ and $\alpha \in [\frac{2}{3},1)$.
  If every (induced) subgraph of~$G$ admits an $\alpha$-balanced separator of size at~most~$s$, then $G$ has treewidth at~most~$15 \cdot \frac{\log(2/3)}{\log \alpha} \cdot s$.
\end{lemma}




\subsection{Twin-width and oriented twin-width}\label{sec:tww}

A~\emph{partition sequence} $\mathcal P_n, \ldots, \mathcal P_1$ of an $n$-vertex graph $G$ is such that $\mathcal P_i$ is a partition of $V(G)$ for every $i \in [n]$, $\mathcal P_n = \{\{v\}~|~v\in V(G)\}$, and for every $i \in [n-1]$, $\mathcal P_i$ is obtained from $\mathcal P_{i+1}$ by merging two parts $X, Y \in \mathcal P_{i+1}$ into one: $X \cup Y$.
In particular, note that $\mathcal P_1=\{V(G)\}$.
For every $i \in [n]$, we denote by $\mathcal R(\mathcal P_i)$ the graph with one vertex per part of~$\mathcal P_i$, and an edge between $P \neq P' \in \mathcal P_i$ whenever there are (possibly equal) $u, v \in P$ and $u', v' \in P'$ such that $uu' \in E(G)$ and $vv' \notin E(G)$.
The \emph{twin-width} of~$G$, denoted by $\tww(G)$, is the least integer~$d$ for which $G$ admits a~partition sequence $\mathcal P_n, \ldots, \mathcal P_1$ such that for every $i \in [n]$, the graph $\mathcal R(\mathcal P_i)$ has maximum degree at~most~$d$.

From the definition of twin-width, it can be observed that the twin-width of a~graph is at~least the twin-width of any of its induced subgraphs.
\begin{observation}[\cite{twin-width1}]\label{obs:tww-ind-sub}
  For any graph $G$ and any induced subgraph $H$ of~$G$, $\tww(H) \leqslant \tww(G)$.
\end{observation}

We also define the digraph $\overrightarrow{\mathcal R}(\mathcal P_i)$, also on vertex set $\mathcal P_i$, with an arc from $P \in \mathcal P_i$ to $P' \in \mathcal P_i \setminus \{P\}$ whenever there are $u \neq v \in P$ and some $w \in P'$ such that $uw \in E(G)$ and $vw \notin E(G)$.
Similarly, the \emph{oriented twin-width} of~$G$, denoted by $\otww(G)$, is the least integer~$d$ for which $G$ admits a~partition sequence $\mathcal P_n, \ldots, \mathcal P_1$ such that for every $i \in [n]$, the digraph $\overrightarrow{\mathcal R}(\mathcal P_i)$ has maximum outdegree at~most~$d$.

It is immediate from the definition that the twin-width of any graph is upper bounded by its oriented twin-width.
Rather surprisingly, twin-width and oriented twin-width are functionally equivalent.
\begin{theorem}[\cite{twin-width6}]\label{thm:otww-tww}
  There is some~$c$ such that for every $G$, $\tww(G) \leqslant \otww(G) \leqslant 2^{2^{c \cdot \tww(G)}}$.
\end{theorem}

This simplifies the task of bounding the twin-width of a~class.
One can just upper bound its oriented twin-width.
We see such an example with a~proof of~\cref{thm:tww-of-ur-lw}.

\begin{reptheorem}{thm:tww-of-ur-lw}
The class of finite induced subgraphs of any neat, upward-restricted layered wheel has bounded twin-width.
\end{reptheorem}
\begin{proof}
  Fix a~neat, upward-restricted layered wheel $W$ with rooted tree~$T$.
  As $W$ is upward-restricted there is an integer~$t$ such that for every node $v$ of~$T$, there is a~set $X_v$ of at~most~$t$ ancestors of~$v$ with the property that every non-layer edge of~$W$ with exactly one endpoint among the descendants of~$v$ has its other endpoint in~$X_v$.
  
  For every natural number $n$, let $W_{\leqslant n}$ be the finite subgraph of~$W$ induced by its first $n$ layers.
  By~\cref{obs:tww-ind-sub}, we just need to upper bound the twin-width of~$W_{\leqslant n}$ (independently of~$n$).
  By~\cref{thm:otww-tww}, it is enough to upper bound the oriented twin-width of~$W_{\leqslant n}$.

  We build a~partition sequence $\mathcal S$ of~$W_{\leqslant n}$ as follows.
  For every $i$ from $n$ down to 1, denote by $p_1, \ldots, p_h$ the parents of the nodes in the $i$th layer, from left to right.
  For every $j$ from 1 to $h$, denote by $q_{j,1}, \ldots, q_{j,a_j}$ the children of $p_j$, from left to right.
  For every $j$ from 1 to $h$, for every $k$ from 2 to $a_j$, merge the part containing $\{q_{j,1}, \ldots, q_{j,k-1}\}$ with the part containing~$q_{j,k}$.
  After this is done, for every $j$ from 1 to $h$, merge the part containing $\{q_{j,1}, \ldots, q_{j,a_j}\}$ with the part~$\{p_j\}$.
  Note that, as $W$ is neat, this well-defines a~partition sequence of~$W_{\leqslant n}$. 

  Take any partition $\mathcal P$ of $\mathcal S$, and any part $P \in \mathcal P$.
  We claim that the outdegree of $P$ in $\overrightarrow{\mathcal R}(\mathcal P)$ is at~most~$t+3$.
  Note that $P$ has either a~unique topmost vertex in $T$, or all its topmost vertices in $T$ share the same parent.
  In the former case, let $v$ be this unique topmost vertex.
  Part $P$ has at~most~$t+3$ outneighbors in $\overrightarrow{\mathcal R}(\mathcal P)$: at most one part to its left, at most two parts to its right, and at most one part per vertex in~$X_v$.
  In the latter case, let $v$ be the shared parent.
  Part $P$ has again at~most~$t+3$ outneighbors in $\overrightarrow{\mathcal R}(\mathcal P)$: at most one part to its left, at most one part to its right, and at most one part per vertex in~$X_v \cup \{v\}$.
\end{proof}

\section{Chordal trigraphs and their tree representations}
\label{sec:trigraphs}

It will be convenient to work with \emph{trigraphs}.\footnote{Here the trigraphs and red edges are unrelated to twin-width.}
Trigraphs are graphs with two disjoint edge sets: black edges and red edges.
Red edges should be viewed here as virtual edges, and will not be part of the eventual construction.
They are however useful for the inductive definitions and in the proofs.

A~\emph{trigraph} is a triple $(V, E_B, E_R)$, where $E_B$ and $E_R$ are disjoint subsets of $V \choose 2$ called the set of \emph{black} (or \emph{real}) and \emph{red} (or \emph{virtual}) edges, respectively.  
For any $u, v \in V(G)$, the \emph{adjacency type} of $uv$ is defined as `black' if $uv \in E_B(G)$, or `red' if $uv \in E_R(G)$, or `non-edge' otherwise.
The graph $(V, E_B\cup E_R)$, denoted by $\mathcal T(G)$, is called the \emph{total graph} of~$G$, and $(V, E_B)$ the \emph{real} graph of $G$.
A~trigraph is said to be \emph{chordal} if its total graph is chordal.

Two trigraphs $G$ and $H$ are \emph{isomorphic} if there is a~bijection $f$ from $V(G)$ to $V(H)$ such that $uv$ (in $G$) and $f(u)f(v)$ (in $H$) have the same adjacency type, for every $u, v \in V(G)$.
When $X\subseteq V(G)$, we denote by $G[X]$ \emph{the subtrigraph of~$G$ induced by $X$}, that is, the trigraph $(X, E_B(G) \cap {X \choose 2}, E_R(G) \cap {X \choose 2})$.
We say that that a~trigraph $H$ is an \emph{induced subtrigraph} of a~trigraph $G$ if for some set $X\subseteq V(G)$, $H$ is isomorphic to $G[X]$.
A trigraph is \emph{$H$-free} if $H$ is not an induced subtrigraph of $G$.

When $T$ is a rooted tree, we use the words \emph{parent}, \emph{child}, \emph{ancestor}, and \emph{descendant} with their usual meaning. 
A~node is its own ancestor and descendant.
We use the notation $\parent(v)$ for the parent of a~non-root node $v$, $\anc(v)$ for the set of ancestors of $v$, and $\desc(v)$ for the set of descendants of $v$.
Nodes $u, v$ are \emph{siblings} in $T$ if they have the same parent.
The \emph{depth} of a~node is its distance to the root.

A~\emph{tree representation} of a~chordal trigraph $H'$ with total graph $H := \mathcal T(H')$ is a~rooted tree $T$ on the same vertex set such that:
\medskip
\begin{compactenum}
\item \label{tr:1} for every $v\in V(H)$, $N_H(v) \cap \anc(v)$ is a clique in $H$,
\item  \label{tr:2} for every non-root vertex $v\in V(H)$,
  $$N_H(v) \subseteq \desc(v) \cup (\anc(v) \cap N_H[\parent(v)]),~\text{and}$$ 
\item  \label{tr:3} if $u$ and $v$ are siblings in $T$, then $u$ and $v$ have a~common ancestor $x$ such that $ux$ and $vx$ have distinct adjacency types in~$H'$.
\end{compactenum}

\begin{figure}
  \begin{center}
    \includegraphics[width=8cm]{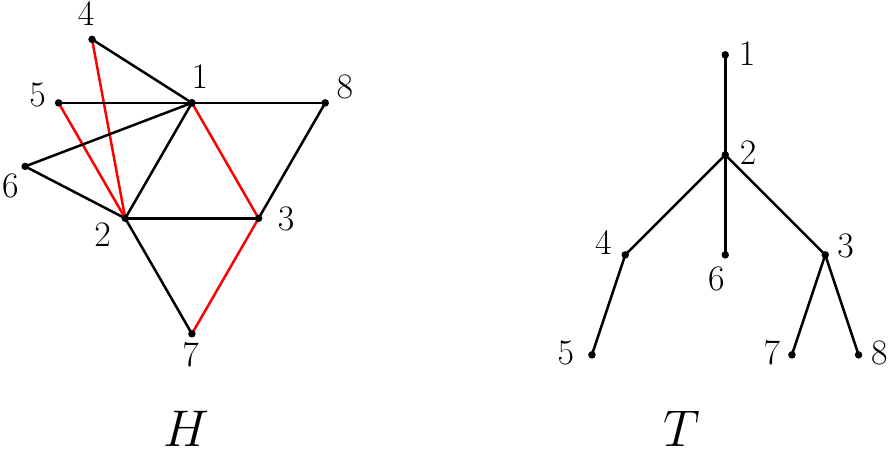}
  \end{center}
  \caption{A chordal trigraph $H$ and a tree representation $T$ of $H$.}
   \label{f:rep}
\end{figure}

\medskip
Condition~\ref{tr:3}, the only condition depending on $H'$ rather than merely on its total graph, is called \emph{sibling condition}.
It imposes that siblings cannot have the exact same adjacency types toward their (common) strict ancestors.
An example is represented in \cref{f:rep}. 
Observe that vertices 4 and 5 are adjacent in $T$ but not in $H$; so in general, $T$ is \emph{not} a~spanning tree of~$H$.
More classical representations of chordal graphs would make vertices 4 and 5 siblings.
Here they cannot because of our sibling condition.
Nonetheless, the following holds.

\begin{lemma}\label{lem:tr}
  Every finite chordal trigraph admits a tree representation. 
\end{lemma}

\begin{proof}
  Fix any finite chordal trigraph $H'$ and set $H := \mathcal T(H')$.
  We prove the property by induction on $|V(H)|$.
  If $H$ (or $H'$) has a~single vertex, then a~single-node tree represents~$H'$.
  
  Suppose now that $|V(H)| \geq 2$.
  By \cref{fact:simpG}, there exists a simplicial vertex $v$ in $H$.
  Let $T'$ be a tree representation of $H' - \{v\} := H'[V(H') \setminus \{v\}]$ obtained from the induction hypothesis. 
  If $v$ is isolated in $H$, then define $T$ by adding $v$ as a~child of any (fixed) leaf in~$T'$.

  Otherwise, let $u$ be \emph{the} deepest vertex in $T'$ of $N_H(v)$.
  Since $v$ is a simplicial vertex, $N_H(v)$ is a~clique, so $u$ is well-defined.
  Indeed by Condition~\ref{tr:2}, every pair of adjacent vertices in $\mathcal T(H' - \{v\})$ are in an ancestor--descendant relation in~$T'$, and every clique of $\mathcal T(H' - \{v\})$ lies in a~single branch of~$T'$.
  
  If $u$ has no child $w$ in $T'$ such that for every strict ancestor $x$ of $w$, $wx$ and $vx$ have the same adjacency type in $H'$, then define $T$ by adding $v$ as a~child of~$u$ in $T'$. 
  Otherwise, such a~vertex $w$ exists, and it is unique by the sibling condition.
  Let $w'$ be a~deepest descendant of $w$ (possibly $w$ itself) such that for every strict ancestor $x$ of $w'$ in $T'$, we have that $w'x$ and $vx$ have the same adjacency type in $H'$.
  Then we build $T$ from $T'$ by adding $v$ as a child of~$w'$.

  In all cases, one can easily check that $T$ is indeed a~tree representation of~$H'$.
\end{proof}

A~\emph{chordal completion} of a~graph $G$ is any chordal graph $H$ such that $V(H) = V(G)$ and $E(G) \subseteq E(H)$.
It is known that the treewidth of a~graph $G$ is the minimum of the clique number of~$H$ minus one, overall chordal completions $H$ of~$G$. 
A~chordal completion of a~graph $G$ will also be thought as a~chordal trigraph $H$ such that $V(H) = V(G)$ and $E_B(H) = E(G)$.

\section{The new upward-restricted layered wheel}
\label{sec:construction}

We now construct the promised layered wheel $W_t$, in order to prove \cref{thm:main-bdd-tw}.
For that, we build a~countably infinite trigraph $G_t$ parameterized by a positive integer~$t$.

\medskip
\textbf{Construction of $\bm{G_t}$, $\bm{T_t}$, and $\bm{W_t}$.}
We now inductively build the trigraph $G_t$ and a~tree $T_t$ on the same vertex set as~$G_t$.
Note that $T_t$ is a graph (not a~trigraph), and is \emph{not} a~subgraph of $\mathcal T(G_t)$.

The vertex set of $G_t$ is partitioned into infinitely many layers $L_0, L_1, L_2, \ldots$, each inducing a~path of real edges of finite size.
The first layer consists of a~single vertex, which is the root of~$T_t$.
For every integer $i \geqslant 0$, let $L_{\leqslant i} := \bigcup_{0 \leqslant j \leqslant i} L_j$ and assume that $G_t[L_{\leqslant i}]$ and $T_t[L_{\leqslant i}]$ are
defined.
The induced path $L_{i+1}$ is built as follows.

For every $u \in L_i$ taken from left to right, where the first (resp.~last) vertex introduced in~$L_i$ is leftmost (resp.~rightmost), initialize $\children(u)$ to the empty set.
For every ordered pair $(B,R)$ of disjoint subsets of \[N^\uparrow[u] := (N_{\mathcal T(G_t)}(u) \cap L_{\leqslant i-1})\cup\{u\}\] such that $|B \cup R| \leqslant t$, append a vertex $v_{B,R}$ to $L_{i+1}$ at its rightmost end, putting the new edge in $E_B(G_t)$.
Add $v_{B, R}$ to the set $\children(u)$.
In $G_t$, add a black edge between $v_{B,R}$ and every vertex of~$B$ and a red edge between $v_{B,R}$ and every vertex of~$R$.
Every vertex of $\children(u)$ is a~child of~$u$ in $T_t$.
This finishes the inductive definitions of the trigraph $G_t$ and the tree $T_t$.

The layered wheel $W_t$ is the real graph of~$G_t$ (with rooted tree $T_t$), and $\mathcal C_t$ is the class of finite induced subgraphs of~$W_t$.
Note that $\mathcal T(G_t)$ is also a layered wheel with rooted tree $T_t$, and a~spanning supergraph of~$W_t$. In Figures~\ref{fig:new-lw-1}~and~\ref{fig:new-lw-2}, we illustrate the first 3 layers of $T_t, G_t$ and $W_t$ when $t = 1$, respectively $t \geq 2$.

\begin{figure}[h!]
  \centering
  \resizebox{400pt}{!}{
    \includegraphics[width=1\linewidth]{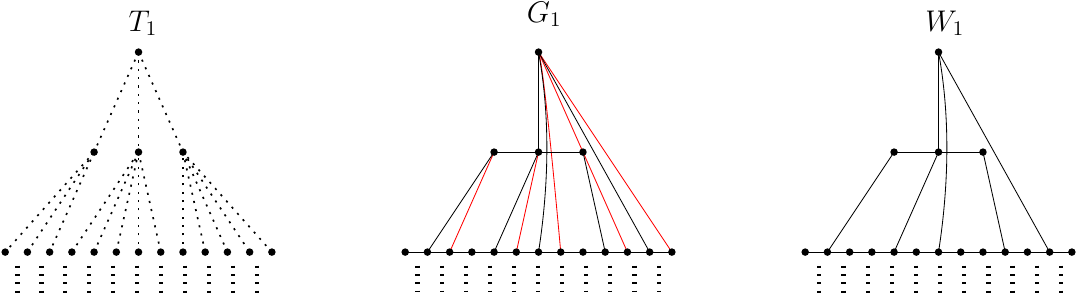}
  }
    \caption{First three layers of $T_1, G_1, W_1$.}
    \label{fig:new-lw-1}
\end{figure}

\begin{figure}
  \centering
  \resizebox{350pt}{!}{
    \includegraphics[width=1\linewidth]{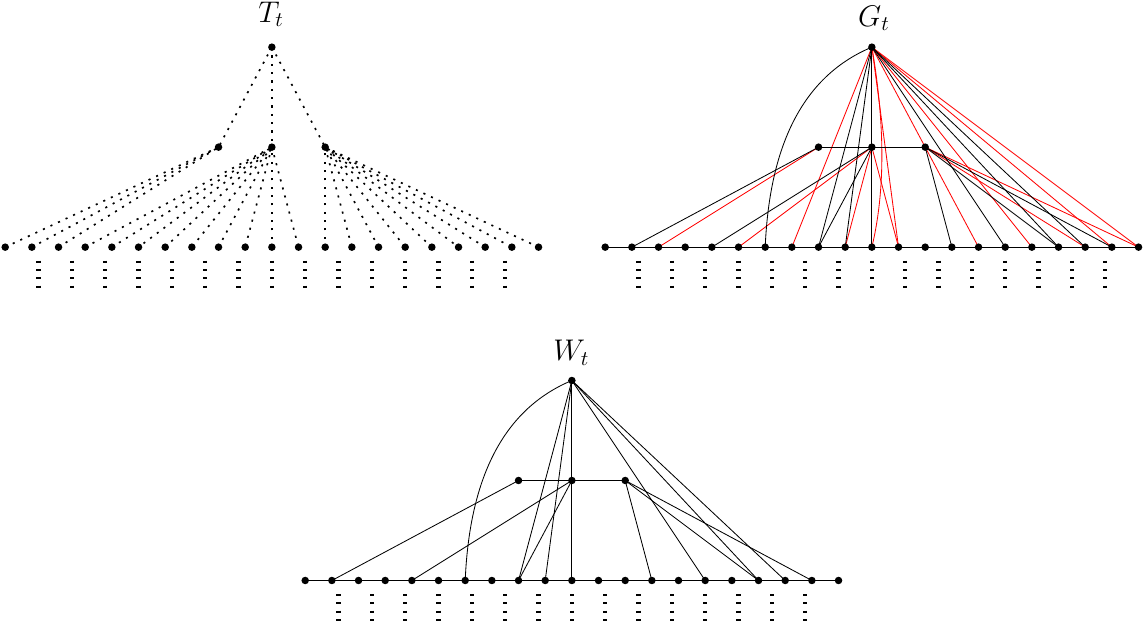}
  }
    \caption{First three layers of $T_t, G_t, W_t$ when $t \geq 2$.}
    \label{fig:new-lw-2}
\end{figure}

We make some observations.
\begin{observation}\label{lem:upward-degree}
  For every $u \in V(G_t)$, $|N^\uparrow[u]| \leqslant t+1$.
\end{observation}

\begin{observation}\label{obs:upward-restricted}
  The layered wheels $\mathcal T(G_t)$, and hence $W_t$, with rooted tree $T_t$ are upward-restricted.
\end{observation}
\begin{proof}
  For every $u \in V(G_t)$, the required set $X_u$ is simply $N^\uparrow[u]$ (with $|X_u| \leqslant t+1$). 
\end{proof}

\begin{observation}\label{obs:Ct-clique-number}
  Every graph of $\mathcal C_t$ has clique number at~most~$t+1$.
\end{observation}
\begin{proof}
  Let $K$ be a~(finite) clique in $W_t$, and let $u \in K$ be a~deepest vertex of~$K$, i.e., no vertex of $K$ is in a~layer of strictly larger index.
  Thus $K \setminus N^\uparrow[u]$ is empty or consists of exactly one sibling $u'$ of~$u$.
  In the former case, we are done, by~\cref{lem:upward-degree}.
  In the latter case, $|K \setminus \{u, u'\}| \leqslant t-1$, as otherwise $u$ and $u'$ would have the exact same neighbors and non-neighbors among their (common) ancestors, which is ruled out by our construction.
\end{proof}

The layered wheels $\mathcal T(G_t)$ and $W_t$, with rooted tree $T_t$, are $3^{t+1}$-bounded.

\begin{observation}\label{obs:degree-Tt}
  Every node of~$T_t$ has fewer than~$3^{t+1}$ children.
\end{observation}
\begin{proof}
  Indeed the number of ordered pairs $(B,R)$ such that $B$ and $R$ are disjoint subsets of a~set of size at~most~$t+1$ with $|B \cup R| \leqslant t$ is at~most~$\max(3^{t+1}-2^{t+1},3^t) < 3^{t+1}$. 
\end{proof}

A~\emph{downward path} from $u \in V(G_t)=V(T_t)$ is an infinite path $u_1 u_2 \dots$ in $T_t$ such that $u_1=u$ and for every $i>1$, $u_{i-1}$ is the parent of $u_i$ in $T_t$ (note that a~downward path is not necessarily a path in~$G_t$).
The~\emph{upward path} from $u \in V(G_t)$ is the unique path in $T_t$ from $u$ to the root of~$T_t$.
Observe that the union of the upward path from $u$ and any downward path from $u$ contains exactly one vertex in each layer: it is a~branch of $T_t$.

We now show the main technical lemma of this section: The finite $H$-free induced subgraphs of $G_t$ satisfy the bounded-branch property for every chordal trigraph~$H$ whose total graph has clique number at~most~$t+1$.

\begin{lemma}[Bounded-branch property]\label{lem:bounded-ray-tw}
  For every chordal trigraph $H$ with $\mathcal T(H)$ of clique number at~most~$t+1$, every set $X \subseteq V(G_t)$ such that $G_t[X]$ is $H$-free, and every $u \in V(G_t)$, there exists a~downward path $P_u$ from~$u$ satisfying $$|V(P_u) \cap X| \leqslant |V(H)|-1.$$
\end{lemma}

\begin{proof}
  Let $T_H$ be a~tree representation of $H$, as defined in \cref{sec:trigraphs}.
  We number $v_1, \ldots, v_h$ the vertices of trigraph~$H$ such that for all $i \in [h]$, $v_i$ is a~leaf of $T_H[\{v_1, \ldots, v_i\}]$.
  Fix any set $X \subseteq V(G_t)$ such that $G_t[X]$ is $H$-free, and any $u \in V(G_t)$.

  Let us prove for every integer $k \in [h]$, the following disjunction, which we call $\mathcal P_k$.
  \begin{compactenum}
  \item \label{disj:1} There exists a subtree $S_k$ of $T_t$, rooted at $u$, such that all the following facts hold:
    \begin{compactitem}
    \item $G_t[X \cap V(S_k)]$ is isomorphic to $H[\{v_1, \dots, v_k\}]$,
    \item denoting by $v'_i$ the image of $v_i$ in the isomorphism from $H[\{v_1, \dots, v_k\}]$ to $G_t[X \cap V(S_k)]$, for any pair of vertices $v_i, v_j$ with $i, j \in [k]$, $v_i$ is an ancestor of $v_j$ in $T_H$ if and only if $v'_i$ is an ancestor of $v'_j$ in $S_k$, and
    \item for any pair of vertices $v_i'$ and $v_j'$ such that $v_i$ is the parent of $v_j$ in $T_H$, if $x \neq v_i'$ is on the path from $v_i'$ to $v_j'$ in $S_k$, then for every $w \in \anc_{T_t}(x) \cap \{v'_1, \dots, v'_k\}$, $xw$ and $v_j'w$ have the same adjacency type in~$H$,
    \end{compactitem}
  \item \label{disj:2} or there exists a downward path $P_u$ from~$u$ such that $|V(P_u) \cap X| \leqslant k-1$.
  \end{compactenum}
  \medskip
  The lemma indeed follows because only the second outcome of $\mathcal P_h$ can hold since $G_t[X]$ is $H$-free.
  We now prove $\mathcal P_k$ by induction on $k$.

  Consider the base case $k=1$.
  If there is some descendant $v \in X$ of $u$ in $T_t$, then add a~shortest path from $u$ to such a~vertex $v$.
  This path is the tree $S_1$ and set $v'_1:=v$, which satisfies Disjunct~\ref{disj:1}.
  Otherwise, $u$ has no descendant in $X$, and any downward path $P_u$ from $u$ satisfies Disjunct~\ref{disj:2}.

  Suppose now that $\mathcal P_k$ holds for some $k \in [h-1]$.
  Let us prove $\mathcal P_{k+1}$.
  If Disjunct~\ref{disj:2} of~$\mathcal P_k$ holds, then the same downward path~$P_u$ from~$u$ satisfies Disjunct~\ref{disj:2} of~$\mathcal P_{k+1}$. 
  We thus assume that Disjunct~\ref{disj:1} of~$\mathcal P_k$ holds.
  Since $G_t[X \cap V(S_k)]$ is isomorphic to $H[\{v_1, \dots, v_k\}]$, some vertex $v'_i \in X \cap V(S_k)$ is such that $v_i$ is the parent of $v_{k+1}$ in $T_H$ with $i \in [k]$.

  From the definition of tree representations, for every $v \in V(H)$, $N_{\mathcal T(H)}(v) \cap \anc_{T_H}(v)$ is a clique in $\mathcal T(H)$, and in particular has size at most $t$ (indeed $v$ is fully adjacent to this clique in $\mathcal T(H)$).
  By construction of $G_t$, there is a~downward path $Q$ in $T_t$ from $v'_i$ such that for every $x \in V(Q) \setminus \{v'_i\}$ and for every $j \in [k]$ with $v_j \in \anc_{T_H}(v_{k + 1})$, $xv'_j$ (in~$G_t$) and $v_{k + 1}v_j$ (in~$H$) have the same adjacency type.
  In particular, for any vertex $w$ of $V(Q) \setminus \{v'_i\}$, the trigraph $G_t[\{v'_1, \dots, v'_k, w\}]$ is isomorphic to $H[\{v_1, \dots, v_{k+1}\}]$.
  Moreover, $V(Q) \setminus \{v'_i\}$ does not contain any $v_j'$ with $j \in [k]$, because of the sibling condition in $T_H$, together with the third condition of Disjunct~\ref{disj:2}.

  If $V(Q) \setminus \{v'_i\}$ contains no vertex of $X$, then we build the downward path $P_u$ by taking the union of $Q$ and the (unique) path from $u$ to $v'_i$ in~$S_k$.
  Vertices from $V(P_u) \cap X$ are all in~$S_k$, so $V(P_u) \cap X$ contains at most $k$ vertices.  

  Otherwise, $V(Q) \setminus \{v'_i\}$ contains some vertex in $X$.
  Let $Q'$ be the shortest subpath of $Q$ from $v'_i$ to some vertex $w \in X \setminus \{v'_i\}$.
  We set $v'_{k + 1} := w$, and we build the tree $S_{k+1}$ by adding $Q'$ to~$S_k$.
  One can finally observe that $S_{k + 1}$ satisfies the three conditions of~Disjunct~\ref{disj:2}.
\end{proof}

We recall that $W_t$ is the real graph of~$G_t$, and $\mathcal C_t$, the class of finite induced subgraphs of~$W_t$.
It is easy to see that $\mathcal C_t$ has unbounded treewidth.

\begin{observation}\label{obs:tw-of-Ct}
  $W_t[L_{\leqslant i}]$ has treewidth at~least~$i$.
  In particular, $\mathcal C_t$ has unbounded treewidth.
\end{observation}

\begin{proof}
  The $i+1$ pairwise disjoint connected sets $L_0, L_1, \ldots, L_i$ form a~bramble in $W_t[L_{\leqslant i}]$.
  Hence the treewidth of $W_t[L_{\leqslant i}]$ is at~least~$(i+1)-1=i$.
\end{proof}

We will now prove that for any (finite) $H$ of treewidth at~most~$t$, the finite $H$-free induced subgraphs of $W_t$ have treewidth upper bounded by a~function of~$t$ and $|V(H)|$.   
For that, we use the functional equivalence of treewidth and separation number, and we will rely on the following separator.

\begin{observation}\label{obs:separator}
  For every $u \in V(G_t)$, let $u^-$ be the vertex to the immediate left (in their layer) of the leftmost child of~$u$, and $u^+$, the vertex to the immediate right of the rightmost child of~$u$.
  For any downward paths $P_{u^-},P_{u^+}$ from $u^-, u^+$, respectively, the set
  \[S := N^{\uparrow}[u] \cup N^{\uparrow}[u^-] \cup N^{\uparrow}[u^+] \cup P_{u^-} \cup P_{u^+}\]
  separates in $\mathcal T(G_t)$, and hence in $W_t$, the descendants of $u$ (as well as some descendants of~$u^-$ and $u^+$) from the rest of the graph.
  (If one $u'$ of $\{u^-, u^+\}$ is not defined, simply remove $N^{\uparrow}[u'] \cup P_{u'}$ from the definition of~$S$.)
\end{observation}

\begin{proof}
  By~\cref{obs:upward-restricted}, $\mathcal T(G_t)$, and hence $W_t$, are upward-restricted.
  Therefore, for any $u' \in \{u, u^-, u^+\}$ and any descendant $w$ of $u'$,
  \begin{equation}\label{eq:inclusion}
    N^{\uparrow}[w] \cap \anc_{T_t}(u') \subseteq N^{\uparrow}[u'] \subseteq S.
  \end{equation} 

    Let $D$ be the connected component of $\mathcal T(G_t) - S$ containing the descendants of~$u$. 
    Let $xy$ be an edge with exactly one endpoint in~$D$.
    If $x$ and~$y$ are both descendants of $u^-$ or of~$u^+$, then one of $x$ and~$y$ is in $P_{u^-} \cup P_{u^+} \subseteq S$.
    Otherwise exactly one of $x, y$ is a~descendant of $u$, $u^-$, or~$u^+$, and the other is in $S$ by~Inclusion (\ref{eq:inclusion}).
\end{proof}

The next lemma locates an appropriate node~$u$ on which to apply~\cref{obs:separator}.

\begin{lemma}\label{lem:balanced-vertex}
  For every $n \geq 8$ and every $n$-vertex induced subgraph $G$ of $\mathcal T(G_t)$ (hence of~$W_t$), there is some $u$ of $T_t$ such that there are at least $\frac{n}{8 \cdot 3^{t+1}}$ and at most $\frac{n}{4}$ descendants of $u$ that are in~$V(G)$, and for every node $v \in V(T_t)$ on the layer of $u$, $v$ also has at~most $\frac{n}{4}$ descendants in~$V(G)$.
\end{lemma}

\begin{proof}
  We start from the root of~$T_t$, and move down along edges of~$T_t$ according to the following rule.
  If the current node has at~most~$n/4$ vertices of~$G$ among its descendants in~$T_t$, stop.
  Otherwise, move to a~child of the current node with a~largest number of descendants within $V(G)$, that is, at~least $\left(\frac{n}{4} - 1\right)/3^{t + 1} \geqslant \frac{n}{8 \cdot 3^{t+1}}$, by~\cref{obs:degree-Tt}.
  Thus, the node we stop at has the required property.
\end{proof}

Graph $W_t$ is a~spanning subgraph of~$\mathcal T(G_t)$.
In particular, for every $X \subseteq V(G_t) = V(W_t)$, we have $\tw(W_t[X]) \leqslant \tw(\mathcal T(G_t)[X])$.
We can now conclude.

 \begin{lemma}\label{lem:main}
   For every graph $H$ of treewidth at~most~$t$ and every finite $H$-free induced subgraph $G$ of $W_t$, the treewidth of $G$ is at~most $15 \cdot \frac{\log(2/3)}{\log \alpha} \cdot (3t+2|V(H)|+1)$ with $\alpha := 1-\frac{1}{8 \cdot 3^{t+1}}$.
 \end{lemma}

 \begin{proof}
   Let $X \subseteq V(W_t)$ be the finite set such that $G=W_t[X]$.
   Let $H'$ be a~chordal trigraph whose real graph is~$H$, and whose total graph is a~chordal graph of clique number at~most~$t+1$. 
   As $H$ is not an induced subgraph of~$W_t[X]$, $H'$ is not an induced subtrigraph of~$G_t[X]$.
   
 If $n := |V(G)| < 8$, we are done. Otherwise, let $u$ be the node of $T_t$ obtained by applying \cref{lem:balanced-vertex} to~$G$.
 Let $u^-$ be the vertex to the immediate left (in their layer) of the leftmost child of~$u$, and $u^+$, the vertex to the immediate right of the rightmost child of~$u$.
 Let $P_{u^-}, P_{u^+}$ be the downward paths given by~\cref{lem:bounded-ray-tw} for $H', X, u^-$ and $H', X, u^+$, respectively.
 Let \[S := (N^{\uparrow}[u] \cup N^{\uparrow}[u^-] \cup N^{\uparrow}[u^+] \cup P_{u^-} \cup P_{u^+}) \cap X.\]
 In the case when $u^-$ (resp.~$u^+$) does not exist, simply remove $P_{u^-} \cup N^{\uparrow}[u^-]$ (resp.~$P_{u^+} \cup N^{\uparrow}[u^+]$) from~$S$.
 We have \[|S| \leqslant 3(t+1) + 2(|V(H')|-1) = 3t+2|V(H)|+1\]
 by~\cref{lem:upward-degree} and~\cref{lem:bounded-ray-tw}.

 \begin{figure}[h!]
\centering
  \begin{tikzpicture}[
  vertex/.style={fill,circle,inner sep=0.05cm, minimum width=0.05cm},
  tree-edge/.style={dotted},
  edge/.style={very thick},
  tedge/.style={white, line width=0.6pt, dotted}
  ]
    \def\s{1}
    \def\t{12}
    \pgfmathtruncatemacro\tm{\t-1}
    \pgfmathtruncatemacro\td{\t/2}
    \pgfmathtruncatemacro\tdp{\td+1}

    \node[vertex, label=above:$u$] (u) at (0.5 * \t * \s + 0.5 * \s, 2 * \s) {} ;
    \foreach \i in {1,...,\t}{
      \node[vertex] (u\i) at (\i * \s, \s) {} ;
    }
    \foreach \i/\l in {1/$u^-$,\t/$u^+$}{
      \node[circle,inner sep=0.05cm, minimum width=0.05cm,label=above:\l]  at (\i * \s, \s) {} ;
    }

    \foreach \i [count = \ip from 2] in {1,...,\tm}{
      \draw[edge] (u\i) -- (u\ip) ;
    }
    \foreach \i in {2,5,6,9}{
      \draw[edge] (u) -- (u\i) ;
      \draw[tedge] (u) -- (u\i) ;
    }
    \foreach \i in {3,4,7,8,10,11}{
      \draw[tree-edge] (u) -- (u\i) ;
    }

 \begin{scope}[shift={(u1)}]
  \draw (1.6,-3) -- (0,0) -- (-1.6,-3) -- cycle ;
  \draw[decorate, decoration={snake, amplitude=0.05cm, segment length=4.5mm}] (0,0) -- (0.3,-3) node[midway,left] {\(P_{u^-}\)};
\end{scope}
\begin{scope}[shift={(u12)}]
  \draw (1.6,-3) -- (0,0) -- (-1.6,-3) -- cycle ;
  \draw[decorate, decoration={snake, amplitude=0.06cm, segment length=5mm}] (0,0) -- (-0.1,-3) node[midway, right] {\(P_{u^+}\)};
\end{scope}

\node at (\td * \s + 0.5 * \s, - 0.5 * \s) {$\frac{n}{8 \cdot 3^{t+1}} \leqslant |\desc(u) \cap V(G)| \leqslant \frac{n}{4}$} ;
\node at (1.8 * \s, -1.5 * \s) {$\leqslant \frac{n}{4}$} ;
\node at (\tm * \s + 0.2 * \s, -1.5 * \s) {$\leqslant \frac{n}{4}$} ;

  \end{tikzpicture}
  \caption{Lower bound of $\frac{n}{8 \cdot 3^{t+1}}$ and upper bound of $\frac{3n}{4}$ in the number of vertices of $G$ below the separator~$S$.}
  \label{fig:balanced-sep}
 \end{figure}
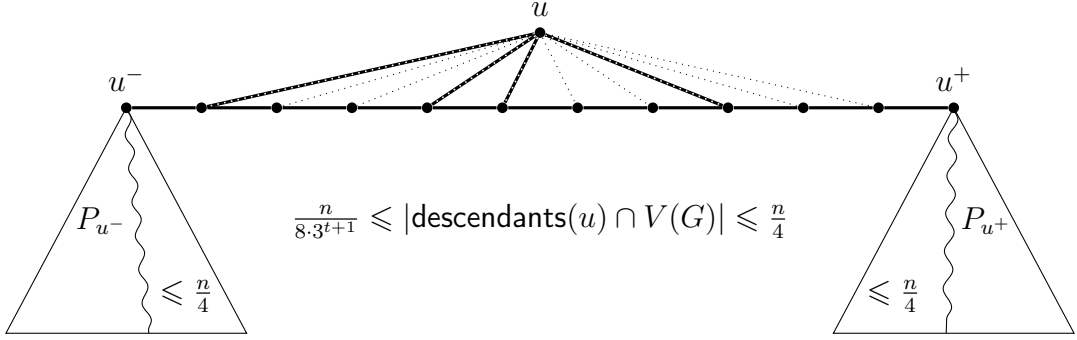
 
 By~\cref{obs:separator}, the descendants of $u$, and part of the descendants of $u^-$ and of $u^+$, are disconnected from the other vertices of $G-S$.
 By \cref{lem:balanced-vertex}, $S$ is an~$\alpha$-balanced separator with $\alpha := 1-\frac{1}{8 \cdot 3^{t+1}}$.
 Indeed there are at~least~$\frac{n}{8 \cdot 3^{t+1}}$ vertices of $G-S$ \emph{under $S$} (i.e., within the descendants of $u$, $u^-$, or $u^+$), and at~most $\frac{n}{4} + \frac{n}{4} + \frac{n}{4}= \frac{3n}{4}$.
 The latter holds because $u^-$ and $u^+$ are below the layer of~$u$ so, in particular, they have at~most~$n/4$ descendants in~$X$; see~\cref{fig:balanced-sep}. 

 Thus $G$ (and its subgraphs) has an~$\alpha$-balanced separator of size at~most~$3t+2|V(H)|+1$, and we conclude by~\cref{lem:tw-sn-convert}.
 \end{proof}

 \cref{lem:main,obs:tw-of-Ct} establish~\cref{thm:main-bdd-tw}, and hence~\cref{thm:main-one-graph}.

 \medskip

 As mentioned in the introduction, a slight modification to the construction of $G_t$ yields a triangle-free layered wheel $W'_t$ whose hereditary closure satisfies~\cref{thm:main-triangle-free}.
 The construction of ${G'_{t}}$, ${T'_{t}}$, and ${W'_{t}}$ proceeds like the previous construction, with the only difference being in the definition of $\textsf{children}(u)$.

For every $u \in L_i$ taken from left to right, where the first (resp.~last) vertex introduced in~$L_i$ is leftmost (resp.~rightmost), initialize $\textsf{children}(u)$ to the empty set.
Let $\mathcal{B}$ be the set of monochromatic red cliques in $G'_{t}[N^\uparrow[u]]$.
For every ordered pair $(B,R)$ of disjoint subsets of $N^\uparrow[u]$ such that $B \in \mathcal{B}$ and $|B \cup R| \leqslant t$, append a vertex $v_{B,R}$ to $L_{i+1}$ at its rightmost end, putting the new edge in $E_B(G'_{t})$.
Add $v_{B, R}$ to the set $\textsf{children}(u)$.
In $G'_{t}$, add a black edge between $v_{B,R}$ and every vertex of~$B$ and a red edge between $v_{B,R}$ and every vertex of~$R$.
Then append a vertex $v'_{B, R}$ to $L_{i+1}$ at its rightmost end, putting the new edge in $E_B(G'_{t})$ and vertex $v'_{B, R}$ in the set $\textsf{children}(u)$.

\medskip

It is clear that many properties of the previous construction are retained or only slightly altered.
In particular, $W'_{t}$ has unbounded treewidth and is upward-restricted (indeed, for every $u \in V(G'_{t})$, we have that $|N^\uparrow[u]| \leqslant t+1$), and every node of~$T'_{t}$ has fewer than~$2 \cdot 3^{t+1}$ children.

\begin{observation}\label{obs:C't-clique-number}
  Every graph of the hereditary closure $\mathcal C'_{t}$ of $W'_t$ is triangle-free.
\end{observation}

\begin{proof}
  Let $u\in V(W'_{t})$.
  By construction, any two descendants of $u$ that are adjacent to $u$ in $W'_{t}$ and in the same layer, are at least at distance two from each other in their layer and hence do not form a triangle.
    Further, all ancestors of $u$ with black edges to $u$ form a~red clique in $G'_{t}$ and hence an independent set in $W'_{t}$.
    It follows that no vertex is adjacent to both endvertices of an edge in $W'_{t}$.
\end{proof}

We also have:

\begin{lemma}[Bounded-branch property]\label{lem:G't-bounded-ray-tw}
  For every chordal trigraph $H$ with $\mathcal T(H)$ of clique number at~most~$t+1$, and whose real graph is triangle-free, every set $X \subseteq V(G'_{t})$ such that $G'_{t}[X]$ is $H$-free, and every $u \in V(G'_{t})$, there exists a~downward path $P_u$ from~$u$ satisfying $$|V(P_u) \cap X| \leqslant |V(H)|-1.$$
\end{lemma}

This follows directly from the proof for~\cref{lem:bounded-ray-tw} by observing that for a chordal trigraph $H$ with $\mathcal T(H)$ of clique number at~most~$t+1$, and whose real graph $H'$ is triangle-free, a~tree representation $T_H$ satisfies that $N_{H'}(v) \cap \anc_{T_H}(v)$ is an independent set and hence forms a monochromatic red clique in $\mathcal T(H)$.

It is then possible to conclude as in~\cref{lem:main}, with some minor alterations; we skip the details, since \cref{thm:main-triangle-free} also follows from~\cref{thm:small-tw-bbp} (which we prove in the next section).

 \section{Key lemmas on abstract layered wheels}

 We finish this paper by proving~\cref{thm:small-tw-bbp,thm:lw-log-tw-ub}, and providing some evidence towards \cref{conj:no-high-girth}.
 The first goal essentially requires adapting~\cref{lem:balanced-vertex}, which uses that the layered wheels $\mathcal T(G_t)$ and $W_t$ are ($3^{t+1}$-)bounded, an assumption which we want to drop. 

 \begin{lemma}\label{lem:balanced-vertex-unbounded}
   For every layered wheel $W$ with rooted tree $T$, for every integer $n \geq 8$ and $n$-vertex induced subgraph $G$ of $W$, either
   \begin{compactitem}
   \item there is some $u \in V(T)$ such that there are at least $\frac{n}{16}$ and at most $\frac{n}{4}$ descendants of $u$ that are in~$V(G)$, and for every node $v \in V(T)$ on the layer of $u$, $v$ also has at~most~$\frac{n}{4}$ descendants in~$V(G)$, or
   \item there is some $u \in V(T)$ and a~child $u^+$ of $u$ such that the total number of descendants in $V(G)$ of all children of $u$ both to the right of the leftmost child of~$u$ and to the left of $u^+$ is at~least $\frac{n}{16}$ and at~most~$\frac{n}{8}$.  
   \end{compactitem}
 \end{lemma}
 \begin{proof}
   Start at the root of $T$, and while the number of descendants in~$V(G)$ of the current node $w$ is at~least~$\frac{n}{4}$ (first stopping condition) move to a child of $w$ with the largest number of descendants in $V(G)$ if this number is at~least~$\frac{n}{16}$, and stop if none exists (second stopping condition).

   If the while loop stops because of the first condition, the first item of the disjunction holds.
   If the while loop stops because of the second condition, let $u \in V(T)$ be the vertex we arrived at.
   Number $u_1, u_2, \ldots, u_h$ the children of $u$ from left to right.
   Let $p \in [h]$ be the smallest index such that the total number of descendants in $V(G)$ of vertices in $\{u_2, \ldots, u_p\}$ is at~least $\frac{n}{16}$ and at~most~$\frac{n}{8}$.
   By construction, for every child~$u_i$ of $u$ the number of descendants in $V(G)$ of~$u_i$ is less than $\frac{n}{16}$.
   Thus, as the number of descendants in $V(G)$ of~$u$ is at~least~$\frac{n}{4}$, $p$ is well-defined and at~most~equal to~$h-1$.
   Vertex $u^+$ is then defined as~$u_{p+1}$, and the second item holds.
 \end{proof}

 We are now equipped for proving~\cref{thm:small-tw-bbp}, which we recall.

 \begin{reptheorem}{thm:small-tw-bbp}
  For every neat, upward-restricted layered wheel $W$, every class of finite graphs satisfying the bounded-branch property in $W$ has bounded treewidth.
\end{reptheorem}
 \begin{proof}
   Fix any neat, $t$-upward-restricted layered wheel $W$ with rooted tree~$T$.
   Let $\mathcal C$ be a~class of finite induced subgraphs of~$W$ satisfying the bounded-branch property for some integer~$h$.
   Fix any $G \in \mathcal C$.
   Let $u \in V(T)$ be the vertex given by applying \cref{lem:balanced-vertex-unbounded} to $T, G$.

   If $u$ satisfies the first item of~\cref{lem:balanced-vertex-unbounded}, we note that \cref{obs:separator} only uses the fact that the layered wheel is neat (for the downward paths to intersect every layer below their starting point) and upward-restricted.
   We then conclude, following the proof of~\cref{lem:main}, that $G$ has a~$\frac{15}{16}$-balanced separator of size at most~$3(t+1)+2h$.

   Let us thus assume that $u$ and its child $u^+$ satisfy the second item of~\cref{lem:balanced-vertex-unbounded}.
   Let $u^-$ be the leftmost child of~$u$.
   Let $P_{u^-}$ (resp.~$P_{u^+}$) be any downward path in $T$ from $u^-$ (resp.~$u^+$) such that $|V(P_{u^-}) \cap V(G)| \leqslant h$ (resp.~$|V(P_{u^+}) \cap V(G)| \leqslant h$).
   As the layered wheel is neat, these downward paths are infinite.
   Similarly to~\cref{obs:separator}, \[S := N^{\uparrow}[u] \cup N^{\uparrow}[u^-] \cup N^{\uparrow}[u^+] \cup P_{u^-} \cup P_{u^+} = N^{\uparrow}[u] \cup P_{u^-} \cup P_{u^+}\]
   separates in $G$ the descendants of children between $u^-$ and $u^+$ in the left-to-right order to the rest of the graph.
   Moreover, this number of descendants is at~least $\frac{n}{16}$ and at~most~$\frac{n}{8} + \frac{2n}{16} = \frac{n}{4}$.
   This means that $G$ has a~$\frac{15}{16}$-balanced separator of size at most~$t+1+2h$.

   We conclude by~\cref{lem:tw-sn-convert} that $\tw(G) \leqslant 15 \cdot \frac{\log(2/3)}{\log(15/16)} \cdot (3(t+1)+2h)$.
 \end{proof}

 We can further adapt the proof of~\cref{thm:small-tw-bbp} to show that finite induced subgraphs of neat, upward-restricted layered wheels have treewidth at~most logarithmic in their number of vertices.
 
 \begin{reptheorem}{thm:lw-log-tw-ub}
   For every neat, $t$-upward-restricted layered wheel $W$ on rooted tree~$T$, every $n$-vertex induced subgraph of $W$ has treewidth at~most~\[15 \cdot \frac{\log(2/3)}{\log(15/16)} \cdot (3(t+1)+2 h \log n)=O(t+\log n),\]
   for some finite integer $h$ depending only on $T$.
 \end{reptheorem}
 \begin{proof}
   In the proof of~\cref{thm:small-tw-bbp}, every time a downward path from a~specific vertex $v$ is added to a~balanced separator, choose a~downward path from $v$ that contains as few vertices of $G$ as possible.
   By definition of a~layered wheel, there is an integer $h$ such that $T$ has no path of length $h$ only consisting of degree-2 vertices.
   Thus, from any node $v$ with fewer than~$n$ descendants in $V(G)$, there is a~downward path from $v$ containing at~most~$h \log n$ vertices of~$G$, by greedily choosing a~branch with fewer descendants in~$V(G)$.
 \end{proof}

 On the other hand, the class of the finite induced subgraphs of any proper, bounded layered wheel has at~least logarithmic treewidth. 

 \begin{observation}\label{obs:lw-log-tw-lb}
   For every proper, $t$-bounded layered wheel $W$ with rooted tree~$T$, and every positive integer $n'$, $W$ admits an $n$-vertex induced subgraph $G$ such that $n \geqslant n'$ and \[\tw(G) \geqslant \frac{\log n}{\log t} - 1.\] 
 \end{observation}
 \begin{proof}
   We can assume that $t \geqslant 2$, as every proper 1-bounded layered wheel is a~clique.
   Let $i$ be the smallest integer such that the subgraph $G$ induced by the first $i + 1$ layers of $W$ has at least $n'$ vertices.
   By~\cref{obs:tw-of-Ct}, which only uses the fact that the layered wheel is proper, $\tw(G) \geqslant i$.
   As the layered wheel is $t$-bounded, $n \leqslant t^{i + 1}$, so $i \geqslant \log n / \log t - 1$.
 \end{proof}

 By~\cref{thm:lw-log-tw-ub,obs:lw-log-tw-lb} the class of finite induced subgraphs of any proper, neat, bounded, upward-restricted layered wheel has logarithmic treewidth.

We finish with the following lemma, which shows that a construction answering \cref{conj:no-high-girth} negatively cannot come from the realm of (proper) layered wheels:

\begin{lemma} \label{lem:no-high-girth-wheels}
Let $H$ be a graph of girth at least $5$ and minimum degree at least $4$. Then for every $t \in \mathbb N$, every proper layered wheel $W$ with girth at least 5 contains an $H$-free induced subgraph of treewidth at least $t$. 
\end{lemma}

\begin{proof}
Let $L_0, L_1, \dots$ be the layers of $W$. For $t \in \mathbb N$, we recursively construct a set $X_t$ consisting of $t$ layers of $W$ (and thus with $\tw(W[X_t]) \geq t$) such that $W[X_t]$ is $H$-free.

Begin with $X_1 = L_0 =: L_{i_1}$, and suppose $X_k = L_{i_1} \cup \dots \cup L_{i_k}$ has been constructed (for some $k \geq 1$). We note that, if for some $i > i_k$, we have that $S \subseteq X_k \cup L_i$ induces a copy of $H$ in $W$, then $S \cap L_i \neq \emptyset$. Moreover, since $H$ has minimum degree at least 4 and $L_i$ induces a path, any vertex $x \in S \cap L_i$ has at least two neighbors in $X_k$. Since $W$ has girth at least 5, no two vertices of $W$ have two neighbors in common. In particular, there are at most ${|X_k| \choose 2}$ vertices in $V(W) \setminus X_k$ with two neighbors in $X_k$ and so at most as many layers $L_i$ such that $W[X_k \cup L_i]$ is not $H$-free. We let $j > i_k$ be smallest such that $W[X_k \cup L_j]$ is $H$-free, and set $X_{k + 1} := X_k \cup L_j$. This concludes the proof.
\end{proof}

\section*{Acknowledgment}

This paper has substantially benefited from \href{https://www.dagstuhl.de/25041}{Dagstuhl Seminar 25041}, ``Solving Problems on Graphs: From Structure to Algorithms'', where three of the authors met to work on the topic.
  
\bibliographystyle{plain}

\end{document}